\documentclass[11pt,reqno]{amsart}

\usepackage{amsmath,amsthm,amsfonts,amssymb}
\numberwithin{equation}{section}

\usepackage{verbatim} 
\usepackage{xspace} 
\usepackage{amssymb}
\usepackage{amsmath}
\usepackage{amsthm} 
\usepackage[latin1]{inputenc}
\usepackage{graphicx}
\usepackage{amscd,amssymb,euscript,graphics}
\makeindex

\newcommand{\Z}{\ensuremath{\mathbb{Z}}}

\newcommand{\R}{\ensuremath{\mathbb{R}}}

\newcommand{\B}{\ensuremath{\mathbb{B}}}
\newcommand{\F}{\ensuremath{\mathbb{F}}}

\newcommand{\mh}{\ensuremath{\mathbb{H}}}

\newtheorem{teo}{Theorem}[section]

\newtheorem{lemma}[teo]{Lemma}
\newtheorem{cor}[teo]{Corollary}
\newtheorem{df}[teo]{Definition}
\newtheorem{example}[teo]{Example}
\newtheorem{remark}[teo]{Remark}

\newcount\theTime
\newcount\theHour
\newcount\theMinute
\newcount\theMinuteTens
\newcount\theScratch
\theTime=\number\time
\theHour=\theTime
\divide\theHour by 60
\theScratch=\theHour
\multiply\theScratch by 60
\theMinute=\theTime
\advance\theMinute by -\theScratch
\theMinuteTens=\theMinute
\divide\theMinuteTens by 10
\theScratch=\theMinuteTens
\multiply\theScratch by 10
\advance\theMinute by -\theScratch

\def\today{{\number\day\space
 \ifcase\month\or
  January\or February\or March\or April\or May\or June\or
  July\or August\or September\or October\or November\or December\fi
 \space\number\year}}

\begin{document}

\title[Asymptotic Traffic Flow in a Hyperbolic Network]{Asymptotic traffic Flow in a Hyperbolic Network: Definition and Properties of the Core}

\author[Yuliy Baryshnikov and Gabriel H. Tucci]
{Yuliy Baryshnikov and Gabriel H. Tucci}

\address{Bell Laboratories Alcatel--Lucent,
Murray Hill, NJ 07974, USA}
\email{ymb@alcatel-lucent.com}
\email{gabriel.tucci@alcatel-lucent.com}

\begin{abstract}
In this work we study the asymptotic traffic flow in Gromov's hyperbolic graphs. We prove that under certain mild hypotheses the traffic flow in a hyperbolic graph tends to pass through a finite set of highly congested nodes. These nodes are called the ``core" of the graph. We provide a formal definition of the core in a very general context and we study the properties of this set for several graphs.
\end{abstract}

\maketitle

\section{Introduction}

\vspace{0.3cm}
\noindent For more than 40 years the scientific community treated large complex networks as being completely random objects. This paradigm has its roots in the work of the mathematicians Paul Erd\"os and Alfred R\'enyi. In 1959, Erd\"os and R\'enyi suggested that such systems could be effectively modelled by connecting a set of nodes with randomly placed links. They wanted to understand what a ``typical" graph with $n$ nodes and $M$ edges looks like. The simplicity of their approach and the elegance of some their theorems generated a great momentum in graph theory, leading to the emergence of a field that focus on random graphs. 

\vspace{0.3cm}
\noindent Recent advances in the theory of complex networks show that even though many characteristic of the Erd\"os and R\'enyi model appear natural in large networks such as the Internet, the World Wide Web and many social networks, this model is not completely appropriate for their study. Over the past few years, there has been growing evidence that many communication networks have characteristics of negatively curved spaces \cite{H0, H1, H2, H3, H4}. From the large scale point of view, it has been experimentally observed that, on the Internet and other networks, traffic seems to concentrate quite heavily on a very small subset of nodes. For these reasons, we believe that many of the characteristics of large communication networks are intrinsic to negatively curved spaces.  More precisely, Gromov in his seminal work \cite{Gromov} introduced the notion of $\delta$--hyperbolicity for geodesic metric spaces. This concept is defined in terms of the $\delta$--slimness of geodesic triangles. We say that a geodesic metric space satisfies the $\delta$--slimness triangle condition, if for any geodesic triangle $ABC$, every side belongs to the $\delta$--vicinity of the union of the other two sides. Examples of such spaces are trees, the standard hyperbolic spaces and more generally, $\mathrm{CAT}(-k)$ spaces.  It is not clear yet how relevant Gromov's hyperbolic spaces are to real networks. However, they share many properties such as exponential growth, dominance of the edge and small diameter, also called the ``small world" property. Many of the recent pictures of the Internet and the World Wide Web suggest an hyperbolic nature also, see Figure \ref{internet} for a recent map of the Internet \cite{internet}.
\begin{figure}[!Ht]
  \begin{center}
    \includegraphics[width=6cm]{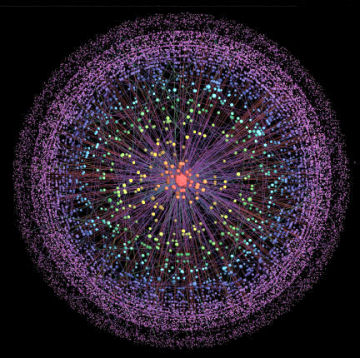}
    \caption{Visualization of the Internet at the AS level. This image was produced by Carmi, Havlin, Kirkpatrick, Shavitt and Shir and was published in \cite{internet}.}
    \label{internet}
  \end{center}
\end{figure}

\vspace{0.3cm}
\noindent We are interested in the ``large scale structure" of the network. Therefore, local properties of the network, such as the valencies, local clustering coefficients, etc., became irrelevant in our analysis. Specifically, we are looking at the properties that survive the ``rescaling" or ``renormalization" of the underlying metric. The main point to notice is that there is a certain ``large scale" structure in these complex networks. Instead of looking at very local geometrical properties we are interested in what happens in the limit when the diameter becomes very large. In this project we are mainly interested in the traffic behaviour as the size of the network increases.

\vspace{0.3cm}
\noindent In Section \ref{prelim}, we review the concept of Gromov's hyperbolic space and present some of the important examples and properties. We also recall the construction of the boundary of an hyperbolic space and its visual metric. In Section \ref{tree}, we study the traffic phenomena in a general locally finite tree. We prove that a big proportion of the total traffic passes through the root. In Section \ref{core}, we formally define the concept of the ``core" of a network and we provide some examples. Finally, in Section \ref{hyper} we characterize the ``core" in a Gromov's hyperbolic space. 

\vspace{0.3cm}
\noindent
{\it Acknowledgement:} We would like to thank Iraj Saniiee for many helpful discussions and comments. This work was supported by AFOSR Grant No. FA9550-08-1-0064.

\section{Preliminaries}\label{prelim} 

\noindent In this Section we review the notion of Gromov's $\delta$--hyperbolic space as well as some of the basic properties, theorems and constructions.

\subsection{$\delta$--Hyperbolic Spaces}
\noindent There are many equivalent definitions of Gromov's hyperbolicity but the one we use is that triangles are {\it slim}.

\begin{df}
Let $\delta>0$. A geodesic triangle in a metric space $X$ is said to be $\delta$--slim if each of its sides is contained in the $\delta$--neighbourhood of the union of the other two sides. A geodesic space $X$ is said to be $\delta$--hyperbolic if every triangle in $X$ is $\delta$--slim. 
\end{df}

\noindent It is easy to see that any tree is $0$-hyperbolic. Other examples of hyperbolic spaces include, any finite graph, the fundamental group of a surface of genus greater or equal than 2, the classical hyperbolic space, and any regular tessellation of the hyperbolic space (i.e. infinite planar graphs with uniform degree $q$ and $p$--gons as faces with $(p-2)(q-2)>4$).

\begin{df} (Hyperbolic Group) A finitely generated group $\Gamma$ is said to be word--hyperbolic if there is a finite generating set $S$ such that the Cayley graph $C(\Gamma,S)$ is $\delta$--hyperbolic with respect to the word metric for some $\delta$.
\end{df}

\noindent It turns out that if $\Gamma$ is a word hyperbolic group then for any finite generating set $S$ of $\Gamma$ the corresponding Cayley graph is hyperbolic, although the hyperbolicity constant depends on the choice of the generator $S$.

\subsection{Boundary of Hyperbolic Spaces}

\vspace{0.3cm}
\noindent We say that two geodesic rays $\gamma_{1}:[0,\infty)\to X$ and $\gamma_{2}:[0,\infty)\to X$ are equivalent and write $\gamma_{1}\sim\gamma_{2}$ if there is $K>0$ such that for any $t\geq 0$
$$
d(\gamma_{1}(t),\gamma_{2}(t))\leq K.
$$
It is easy to see that $\sim$ is indeed an equivalence relation on the set of geodesic rays. Moreover, two geodesic rays $\gamma_{1},\gamma_{2}$ are equivalent if and only if their images have finite Hausdorff distance. The Hausdorff distance is defined as the infimum of all the numbers $H$ such that the images of $\gamma_{1}$ is contained in the $H$--neighbourhood of the image of $\gamma_{2}$ and vice versa.

\vspace{0.3cm}
\noindent The boundary is usually defined as the set of equivalence classes of geodesic rays starting at the base--point, equipped with the compact--open topology. That is to say, two rays are ``close at infinity'' if they stay close for a long time. We make this notion precise.

\begin{df}(Geodesic Boundary)
Let $(X,d)$ be a $\delta$--hyperbolic metric space and let $x_0\in X$ be a base--point. We define the relative geodesic boundary of $X$ with respect to the base--point $x_0$ as 

\begin{equation}
 \partial{X}:=\{[\gamma]\,\,:\,\gamma:[0,\infty)\to X \,\,\text{is a geodesic ray with} \,\,\gamma(0)=x_0\}.
\end{equation}
\end{df}

\vspace{0.3cm}
\noindent It turns out that the boundary has a natural metric.

\begin{df}
Let $(X,d)$ be a $\delta$--hyperbolic metric space. Let $a>1$ and let $x_{0}\in X$ be a base--point. We say that a metric $d_{a}$ on $\partial X$ is a visual metric with respect to the base point $x_{0}$ and the visual parameter $a$ if there is a constant $C>0$ such that the following holds:

\begin{enumerate}
\item The metric $d_{a}$ induces the canonical boundary topology on $\partial X$.
\item For any two distinct points $p,q\in\partial X$, for any bi-infinite geodesic $\gamma$ connecting $p,q$ and any $y\in\gamma$ with  $d(x_{0},\gamma)=d(x_{0},y)$ we have that
$$\frac{1}{C}a^{-d(x_{0},y)}\leq d_{a}(p,q)\leq Ca^{-d(x_{0},y)}.$$
\end{enumerate}
\end{df}

\begin{teo}(\cite{Coor}, \cite{Harper})\label{visual_teo}
Let $(X,d)$ be a $\delta$--hyperbolic metric space. Then:
\begin{enumerate}
\item There is $a_{0}>1$ such that for any base point $x_{0}\in X$ and any $a\in (1,a_{0})$ the boundary $\partial X$ admits a visual metric $d_{a}$ with respect to $x_{0}$.
\item Suppose $d'$ and $d''$ are visual metrics on $\partial X$ with respect to the same visual parameter $a$ and the base points $x_{0}'$ and $x_{0}''$ accordingly. Then $d'$ and $d''$ are Lipschitz equivalent, that is there is $L>0$ such that
$$d'(p,q)/L\leq d''(p,q)\leq Ld'(p,q)\hspace{0.5cm}\text{for any}\,\,p,q\in\partial X.$$ 
\end{enumerate}
\end{teo}

\vspace{0.2cm}
\noindent The metric on the boundary is particularly easy to understand when $(X,d)$ is a tree. In this case $\partial X$ is the space of ends of $X$. The parameter $a_{0}$ from the above proposition is $a_{0}=\infty$ here and for some base point $x_{0}\in X$ and $a>1$ the visual metric $d_{a}$ can be given by an explicit formula:
$$
d_{a}(p,q)=a^{-d(x_{0},y)}
$$
for any $p,q\in\partial X$ where $[x_{0},y]=[x_{0},p)\cap [x_{0},q)$ so that $y$ is the bifurcation point for the geodesic rays $[x_{0},p)$ and $[x_{0},q)$.

\vspace{0.2cm}
\noindent Here are some other examples of boundaries of hyperbolic spaces (for more on this topic see \cite{Coor, Harper, Gromov}.)

\begin{example}
\begin{enumerate}
\item If $X$ is a finite graph then $\partial X=\emptyset$.
\item If $X$ is infinite cyclic group then $\partial X$ is homeomorphic to the set $\{0,1\}$ with the discrete topology.
\item If $n\geq 2$ and $X=\F_{n}$, the free group of rank $n$, then $\partial X$ is homeomorphic to the space of ends of a regular $2n$--valent tree, that is to a Cantor set.
\item Let $S_{g}$ be a closed oriented surface of genus $g\geq 2$ and let $X=\pi_{1}(S_{g})$. Then $X$ acts geometrically on the hyperbolic plane $\mh^{2}$ and therefore the boundary is homeomorphic to the circle $S^{1}$. 
\item Let $M$ be a closed $n$--dimensional Riemannian manifold of constant negative sectional curvature and let $X=\pi_{1}(M)$. Then $X$ is word hyperbolic and $\partial X$ is homeomorphic to the sphere $S^{n-1}$.
\item The boundary of the classical $n$ dimensional hyperbolic space $\mathbb{H}^n$ is $S^{n-1}$.
\end{enumerate}
\end{example}

\section{Traffic Flow in a Tree}\label{tree}

\noindent In this Section we study the asymptotic geodesic traffic flow in a tree. More specifically, let $\{k_{l}\}_{l=0}^{\infty}$ be a sequence of positive integers with $k_{0}=1$. For each sequence like this we consider the infinite tree $T$ with the property that each element at depth $l$ has $k_{l+1}$ descendants. In other words, the root has $k_{1}$ descendants, each node in the first generation has $k_{2}$ descendants and so on. The root is considered the 0 generation. Let us denote by $T_{n}$ the finite tree generated by the first $n$ generations of $T$, and let $N$ be the number of elements in $T_{n}$. It is clear that
$$
N(n)=1+k_{1}+k_{1}k_{2}+\ldots+k_{1}k_{2}\ldots k_{n}=\sum_{l=0}^{n}{\prod_{i=0}^{l}{k_{i}}}.
$$
\begin{figure}[!Ht]
  \begin{center}
    \includegraphics[width=5cm]{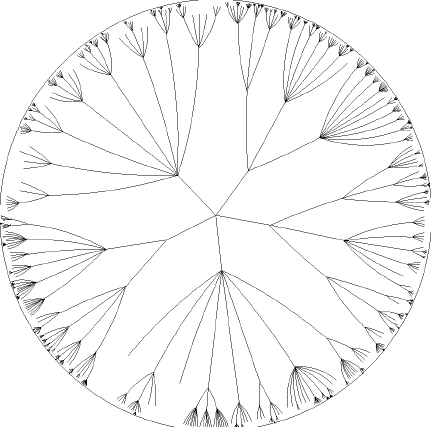}
    \caption{Example of a tree embedded in $\B^{2}$.}
    \label{3-tree}
  \end{center}
\end{figure}
\noindent Assume that very node in the tree is communicating with every other node and they are transmitting a unit flow of data. Therefore, the total traffic in the network is $\frac{N(N-1)}{2}$ (or $N(N-1)$ if we consider bidirectional flows). We are interested in understanding how much of the traffic goes through the root? Also, what is the proportion of this traffic as $n$ goes to infinity? More generally, what is the proportion of the traffic that passes through a node in the $l^{th}$--generation as $n\to\infty$? Giving a node $v$ in the tree let us denote $l_{n}(v)$ the traffic flow passing through $v$ and $p_{n}(v)$ the proportion of traffic through $v$. It is not difficult to compute the total flow through the root of the tree. If we remove the root we have $k_1$ disconnected trees with $\frac{N-1}{k_1}$ elements each. Hence, we see that
\begin{equation}
l_{n}(\mathrm{root})=\binom{k_1}{2}\cdot\Bigg(\frac{N-1}{k_1}\Bigg)^{2}+N-1=\frac{k_1-1}{2k_1}\cdot(N-1)^2+N-1.
\end{equation}
Therefore, the proportion of traffic through the root is 
\begin{equation}
p_{n}(\mathrm{root})=\frac{2\,l_{n}(\mathrm{root})}{N(N-1)}=\frac{k_1-1}{k_{1}}\cdot\frac{N-1}{N}+\frac{2}{N}.
\end{equation}
Hence, $p_{n}(\mathrm{root})$ converges to $1-\frac{1}{k_1}$ as $n\to\infty$. This is the asymptotic proportion of the traffic through the root. We denote this quantity by $ap(\mathrm{root})$. More generally,
\begin{equation}
ap(v):=\liminf_{n\to\infty}{p_{n}(v)}.
\end{equation} 
In the next theorem we compute the asymptotic proportion of the traffic for every node in $T_{n}$.

\begin{teo}
Let $v$ be a node in $T_{n}$ and let $l$ be its depth.  The asymptotic proportion of the traffic through $v$ is 
\begin{equation}
ap(v)=\frac{1}{\beta(l)}\Bigg(2-\frac{1}{\beta(l)}-\frac{1}{\beta(l+1)} \Bigg),
\end{equation}
where 
\begin{equation}
\beta(l):=k_{1}k_{2}\ldots k_{l}
\end{equation}
\end{teo} 

\begin{proof}
Take $v\in T_{n}$ that is at depth $l$. Consider $T_{n}\setminus\{v\}$ the graph formed from $T_{n}$ by deleting the vertex $v$. This graph has $k_{l+1}+1$ connected components with cardinality $s_{1},s_{2},\ldots,s_{k_{l+1}+1}$. Then the flow $l_{n}(v)$ is equal to  
$$
l_{n}(v)=\sum_{i<j}{s_{i}s_{j}}.
$$
Let $c_{n}(v)$ be the number of descendants of $v$ in $T_{n}$. Then as $n\to\infty$
$$
l_{n}(v)\approx \big(N-1-c_{n}(v)\big)\cdot c_{n}(v)+\binom{k_{l+1}}{2}\cdot \Bigg(\frac{c_{n}(v)}{k_{l+1}}\Bigg)^{2}.
$$
It is easy to see that 
$$
c_{n}(v)=k_{l+1}+k_{l+1}k_{l+2}+\ldots+k_{l+1}k_{l+2}\ldots k_{n}.
$$
Since the number of nodes in $T_{n}$ is equal to
$$
N(n)=1+k_{1}+k_{1}k_{2}+\ldots+k_{1}k_{2}\ldots k_{n}=\sum_{l=0}^{n}{\prod_{i=0}^{l}{k_{i}}}
$$
we see that 
$$
\lim_{n\to\infty}{\frac{c_{n}(v)}{N(n)}}=\frac{1}{k_{1}k_{2}\ldots k_{l}}=\frac{1}{\beta(l)}.
$$
We want to compute $ap(v)=\lim_{n\to\infty}{\frac{2l_{n}(v)}{N(N-1)}}$. We first see that doing some algebra we obtain that 
$$
\lim_{n\to\infty}{\frac{2\big(N-1-c_{n}(v)\big)\cdot c_{n}(v)}{N(N-1)}}=\frac{2}{\beta(l)}\cdot \Bigg(1-\frac{1}{\beta(l)}\Bigg)
$$
and hence,
$$
\lim_{n\to\infty}{\frac{2\binom{k_{l+1}}{2}\cdot \Big(\frac{c_{n}(v)}{k_{l+1}}\Big)^{2}}{N(N-1)}}=\frac{k_{l+1}-1}{k_{l+1}}\cdot\frac{1}{\beta(l)^2}.
$$
Therefore,
\begin{eqnarray*}
ap(v) & = & \frac{2}{\beta(l)}\cdot \Bigg(1-\frac{1}{\beta(l)}\Bigg)+\frac{k_{l+1}-1}{k_{l+1}}\cdot\frac{1}{\beta(l)^2}\\
& = & \frac{1}{\beta(l)}\cdot \Bigg(2-\frac{1}{\beta(l)}-\frac{1}{\beta(l+1)} \Bigg).
\end{eqnarray*}
\end{proof}

\begin{cor}
If the tree is $(k+1)$--regular. Then 
\begin{equation}
ap(v)=\frac{1}{k^{l}}\Bigg(2-\frac{1}{k^{l}}-\frac{1}{k^{l+1}} \Bigg).
\end{equation}
In particular, $ap(\mathrm{root})=1-\frac{1}{k}$.
\end{cor}

\section{Definition of the Core of a Network}\label{core}

\vspace{0.3cm}
\noindent In this Section we define the core of a network. More precisely, let $X$ be an infinite, simple and locally finite graph, and let $x_{0}$ be a node in $X$ considered the root or base--point of the graph. Let 
$$ 
\{x_{0}\}=X_{0}\subset X_{1}\subset X_{2}\subset \ldots\subset X_{n}\subset\ldots\subset X
$$
be a sequence of finite subsets with the properties that: 
\begin{itemize}
\item $\cup_{n\geq 1}{X_{n}}=X$,
\item for every $x\in X_{n}$ and for every geodesic segment $[0,x]$ connecting $0$ and $x$  then every intermediate point belongs to $X_{n}$. 
\end{itemize}

\noindent For each fixed $n$ consider a measure $\mu_{n}$ supported on $X_{n}$. This measure defines a traffic flow between the elements in $X_{n}$ in the following way; given $u$ and $v$ nodes in $X_{n}$ the traffic between these two nodes splits equally between the geodesics joining $u$ and $v$, and is equal to $\mu_{n}(v)\mu_{n}(u)$. Note that the uniform measure determines a uniform traffic between the nodes. Given a subset $A\subset X_{n}$ we denote by $\mathcal{L}_{n}(A)$ the total traffic passing through the set $A$. The total traffic in the network $X_{n}$ is equal to $\mathcal{L}_{n}(X_{n})$.

\vspace{0.3cm}
\noindent Let $N=N(n)=|X_{n}|$ be the number of nodes in $X_{n}$.  Given a node $y\in X_{n}$ and $r>0$ we denote by $B(y,r)$ the ball of center $y$ and radius $r$,
$$
B(y,r)=\{x\in X\,:\,d(x,y)\leq r\}.
$$

\begin{df}
Let $\alpha$ be a number between $0$ and $1$. We say that a point $y\in X$ is in the asymptotic $\alpha$--core if 
\begin{equation}
\liminf_{n\to\infty}{\frac{\mathcal{L}_{n}(B(y,r_{\alpha}))}{\mathcal{L}_{n}(X_{n})}}\geq \alpha
\end{equation}
for some $r_{\alpha}$ independent on $n$.
The set of nodes in the asymptotic $\alpha$--core is denoted by $C_{\alpha}$. The core of the graph is the union of all the $\alpha$--cores for all the values of $\alpha$. This set is denoted by $\mathcal{C}$
\begin{equation}
\mathcal{C} = \cup_{\alpha>0}{\,C_\alpha}.
\end{equation}
\end{df}
\noindent We say that a graph has a core if $\mathcal{C}$ is non-empty.

\vspace{0.2cm}
\noindent Roughly speaking a point $y$ belongs to the $\alpha$--core if there exists a finite radius $r$, independent on $n$, such that the proportion of the total traffic passing through the ball of center $y$ and radius $r$ behaves asymptotically as $\Theta(\alpha N^{2})$ as $N\to\infty$.

\begin{example}
Let $\{k_{l}\}_{l=0}^{\infty}$ be a sequence of positive integers and let $T$ be the associated tree as defined in Section \ref{tree}. Consider the increasing sequence of sets $\{T_{n}\}_{n=1}^{\infty}$ where $T_{n}$ is the truncated tree at depth $n$, and $\mu_{n}$ is uniform measure in $T_{n}$. In the previous Section, we saw that the asymptotic proportion of traffic through the root is $ap(\mathrm{root})=1-\frac{1}{k_1}$. Therefore, this graph has a non-empty core. Moreover, the core is the whole graph! However, for every $\alpha$ the $\alpha$--core is finite.
\end{example}

\vspace{0.3cm}
\begin{example}
Let $X=\Z^{p}$ and let $X_{n}=\{-n,\ldots,-1,0,1,\ldots,n\}^{p}$ with the uniform measure. A simple but lengthy analytic calculation shows us that for every $x\in X$ the traffic flow through $v$ behaves as $O(N^{1+\frac{1}{p}})$ where $N=|X_{n}|$. Hence, if $p\geq 2$ the core of this graph is empty. This is not the case for $p=1$.
\end{example}

\section{Properties of the Core for $\delta$--Hyperbolic Networks}\label{hyper}

\noindent In this Section we prove the existence of a core for Gromov's $\delta$--hyperbolic graphs. More precisely, consider $X$ an infinite, simple and locally finite $\delta$--hyperbolic graph with a fixed base--point $x_{0}$. As an example consider the case where $X$ is a group $\Gamma$ acting isometrically in the hyperbolic space $\B^{n}$. We can always think the group $\Gamma$ as a graph embedded in $\B^{n}$. More specifically, the vertices of this graph are the orbits of the point zero under the group action $\{\gamma(0)\}_{\gamma\in\Gamma}$. The point $0$ is considered the root of this graph. Doing an abuse of notation we denote this graph also by $\Gamma$. We normalize the metric so every two adjacent nodes in this graph are at distance 1. Therefore, given two different nodes $u$ and $v$ the distance $d(u,v)$ between them is the minimum number of hops to go from one to the other.  

\begin{figure}[!Ht]
  \begin{center}
    \includegraphics[width=5cm]{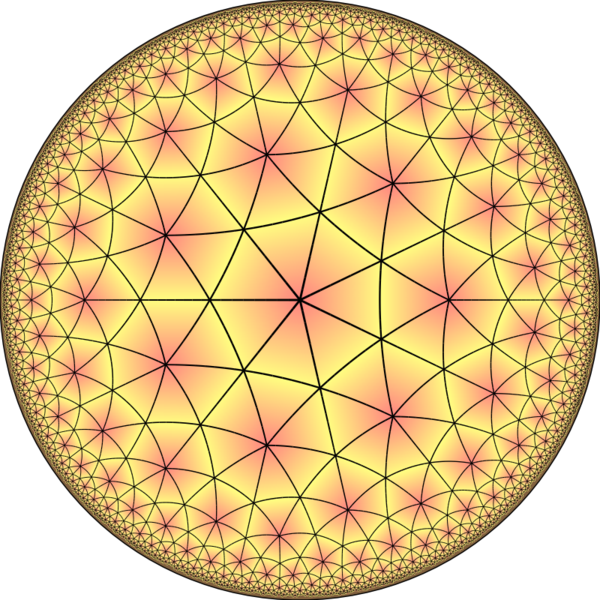}
    \caption{Heptagonal tilling of the hyperbolic space}    
  \end{center}
\end{figure} 
 
\noindent Let 
$$ 
\{0\}=X_{0}\subset X_{1}\subset X_{2}\subset \ldots\subset X_{n}\subset\ldots\subset X
$$
be a sequence of finite subsets as in Section \ref{core}.

\vspace{0.2cm}
\noindent Denote as usual by $\partial X_{n}$ the boundary set of $X_{n}$ in $X$ and recall that a point $y$ belongs to $\partial X_{n}$ if $y\in X_{n}$ and there exists $z\in X\setminus X_{n}$ such that $z\sim y$ ($z$ and $y$ are adjacent). Let $N=N(n)=|X_{n}|$ be the number of elements in $X_{n}$. 

\vspace{0.3cm}
\noindent For each fixed $n$ consider a measure $\mu_{n}$ supported on $X_{n}$ and a traffic measure as we described in Section \ref{core}. We are interested in understanding the asymptotic traffic flow through an element in the network as $n$ grows. More precisely, for each $v$ element in $X$ and $r>0$ consider 
$$
B(v,r):=\{x\in X\,:\,d(v,x)\leq r\},
$$
the ball centred at $v$ and radius $r$. This ball will be eventually contained in $X_{n}$ for every $n$ sufficiently large. Let $\mathcal{L}_{n}(v,r)$ be traffic flow passing through this ball in the network $X_{n}$. We want to study $\mathcal{L}_{n}(v,r)$ as $n$ goes to infinity. As we define in Section \ref{core} an element $v$ is in the asymptotic $\alpha$--core if 
$$
\liminf_{n\to\infty}{\frac{\mathcal{L}_{n}(v,r_{\alpha})}{\mathcal{L}_{n}(X_{n})}}\geq \alpha
$$
for some radius $r_{\alpha}$ independent on $n$. In other words, asymptotically a fraction $\alpha$ of the total traffic passes within bounded distance of $v$.

\vspace{0.3cm}
\noindent Each measure $\mu_{n}$ supported in $X_{n}\subset X\subset\partial X$ defines a probability visual Borel measure $\mu_{n}^{v}$ in the boundary $\partial X$. The way the measure $\mu_{n}^{v}$ is defined is described in the following definition.

\begin{df}
Let $A\subseteq \partial X$ be a Borel subset. For each $a\in A$ consider the set of all sequences $\{x_{k}\}_{k=0}^{\infty}$ such that:
$x_{0}=0$, the sequence is a geodesic ray in $X$ that converges to $a$. These sequences are called rays connecting $0$ with $a$. Let $C_{A}$ be the set of points in $X$ that belong to some ray connecting $0$ with $a$ for some $a\in A$. We define 
\begin{equation}
\mu_{n}^{v}(A):=\frac{\mu_{n}(X_{n}\cap C_{A})}{\mu_{n}(X_{n})}.
\end{equation}
\end{df}

\begin{teo} \label{main1}
Assume that the sequence of visual measures $\{\mu_{n}^{v}\}_{n=1}^{\infty}$ converges to a measure $\mu_{\infty}$.
\begin{itemize}
\item If the limiting measure is concentrated at a point, then $0$ is not in the $\alpha$--core for any $\alpha>0$.
\item If the limiting measure is purely atomic and supported in more than one point, then for some $\alpha_{*}<1$ the $\alpha$--core is empty for all $\alpha>\alpha_{*}$.
\item If the limiting measure is non--atomic then $0$ belongs to the $\alpha$--core for all $\alpha<1$.
\end{itemize} 
\end{teo}

\vspace{0.3cm}

\noindent Before proving our main theorem we recall a result of Bonk and Kleiner (see \cite{Bonk} for more details). This one says that, up to a quasi--isometric embedding, we can assume without loss of generality that our graph is embedded in the hyperbolic space $\mh^{n}$ for some $n$. Consider $\B^{n}$ the Poincar\'e ball as the model for hyperbolic space. We assume that under the previous quasi--isometric embedding $x_{0}$ is mapped to 0. The boundary $\partial \B^{n}$ is homeomorphic to $S^{n-1}$ the unit sphere in $\R^{n}$. For any pair of points in the boundary $\partial\B^{n}=S^{n-1}$ of the hyperbolic space, it is natural to consider the spherical metric. In other words, given $u$ and $v$ in $S^{n-1}$ we can define $d_{S^{n-1}}(u,v)=\theta$ as the angle $[0,\pi]$ between these two points in $S^{n-1}$. On the other hand, let $\gamma_{uv}$ be the bi-infinite geodesic connecting $u$ and $v$ in the Poincar\'e metric and let 
$$
h(u,v):=d_{\B^{n}}(0,\gamma_{uv}).
$$
The next Lemma relates the angle $\theta$ with $h(u,v)$.
\begin{lemma}\label{lemma}
Let $u$ and $v$ in $S^{n-1}$ such that $d_{S^{n-1}}(u,v)=\theta$. Then if $h(u,v):=R$ the following relation holds
\begin{equation}\label{eq_theta}
\sin\big(\theta/2\big)=\frac{1}{\cosh(R)}.
\end{equation}
\end{lemma}
\noindent The proof of this Lemma is elementary and we leave it as an exercise for the reader.

\begin{proof}{of Theorem \ref{main1}:}
Without loss of generality and for notation simplicity let us assume that $X\subset\B^{2}$, where the embedding of $X$ in $\B^{2}$ is quasi--isometric (see \cite{Bonk}). Therefore, geodesics in $X$ are quasi--geodesics in $\B^{2}$. By Proposition 7.2 in \cite{Gromov} there exists a positive constant $C$, such that for every pair of point in $x$ and $y$ in $X$ the geodesic path $\gamma_{xy}$ connecting this two points in $X$ and the geodesic path $\tilde{\gamma}_{xy}$ connecting them in $\B^{2}$ satisfy that the Hausdorff distance 
$$
d_{H}(\gamma_{xy},\tilde{\gamma}_{xy})<C.
$$
For every $r>0$, let $B_{r}$ be the ball centred at $0$ and radius $r$ with the metric of $X$,
$$
B_{X}(0,r):=\{x\in X\,:\,d_{X}(0,x)\leq r\}.
$$
We are interested in the asymptotic behaviour of $\mathcal{L}_{n}(B_{X}(0,r))$ as $n$ increases. Fix $u$ in $S^{1}$, if the bi-infinite geodesic $\gamma_{uv}$ ($v\in S^{1}$) enters $B_{r}$ then the bi-infinite geodesic $\tilde{\gamma}_{uv}$ has to enter $B_{\B^{2}}(0,r+C)$. By Lemma \ref{lemma}, the later happens if and only if the spherical distance between $u$ and $v$ is greater that $\theta_{r+C}$ where
$$
\theta_{r+C}=2\cdot\mathrm{asin}\big(\cosh(r+C)^{-1}\big).
$$
Given $u\in S^{1}$ and $R>0$ we define 
$$
E_{u}^{R}:=\{v\in S^{1}\,:\,d_{S^{1}}(u,v)>2\cdot\mathrm{asin}(\cosh(R)^{-1})\}.
$$
Note that geometrically $E_{u}^{R}$ is a cap opposite to $u$ in $S^{1}$. Given the previous consideration it is clear that 
\begin{equation}
\limsup_{n\to\infty}{\frac{\mathcal{L}_{n}(B_{X}(0,r))}{\mathcal{L}_{n}(X_{n})}}\leq \int_{S^{1}}{\mu_{\infty}(E_{u}^{r+C})\,d\mu_{\infty}(u)}.
\end{equation}
Analogously, for $r>C$ we have 
\begin{equation}
\int_{S^{1}}{\mu_{\infty}(E_{u}^{r-C})\,d\mu_{\infty}(u)}\leq \liminf_{n\to\infty}{\frac{\mathcal{L}_{n}(B_{X}(0,r))}{\mathcal{L}_{n}(X_{n})}}.
\end{equation}
Now we analyze each case separately. Assume that $\mu_{\infty}$ is supported on one point. In this case, it is clear that for every $R>0$
$$
\int_{S^{1}}{\mu_{\infty}(E_{u}^{R})\,d\mu_{\infty}(u)}=0,
$$
and therefore, 
$$
\limsup_{n\to\infty}{\frac{\mathcal{L}_{n}(B_{X}(0,r))}{\mathcal{L}_{n}(X_{n})}}=0,
$$
for every $r>0$.

\vspace{0.3cm}
\noindent Assume that $\mu_{\infty}$ is purely atomic and supported in more that one point. Then there exists a sequence of atoms $\{x_{n}\}_{n}$ in $S^{1}$ with weights $\{p_{n}\}_{n}$. It is clear that for every $R>0$ and every $x\in S^{1}$,
$$
\mu_{\infty}(E_{x}^{R})\leq 1-\mu_{\infty}(x).
$$
Then,
\begin{eqnarray*}
\int_{S^{1}}{\mu_{\infty}(E_{u}^{R})\,d\mu_{\infty}(u)} & = & \sum_{n=1}^{\infty}{\mu_{\infty}(x_{n})\mu_{\infty}(E_{x_{n}}^{R})}\\
& \leq & \sum_{n=1}^{\infty}{\mu_{\infty}(x_{n})(1-\mu_{\infty}(x_{n}))}\\
& = & \sum_{n=1}^{\infty}{p_{n}(1-p_{n})}\\
& = & 1- \sum_{n=1}^{\infty}{p_{n}^{2}}=\alpha_{*}<1.
\end{eqnarray*}
Hence, for every $r>0$
$$
\limsup_{n\to\infty}{\frac{\mathcal{L}_{n}(B_{X}(0,r))}{\mathcal{L}_{n}(X_{n})}}\leq \alpha_{*}<1.
$$
Concluding that the $\alpha$--core is empty for every $\alpha>\alpha_{*}$.

\vspace{0.3cm}
\noindent Finally, assume that $\mu_{\infty}$ is non--atomic. Since the unit circle $S^{1}$ is compact we see that for every $\epsilon>0$ there exists $R$ sufficiently large such that 
$$
\mu_{\infty}(E_{x}^{R})\geq 1-\epsilon,
$$
for every $x\in S^{1}$. Then 
$$
\int_{S^{1}}{\mu_{\infty}(E_{u}^{R+C})\,d\mu_{\infty}(u)}\geq 1-\epsilon,
$$
and hence 
$$
\liminf_{n\to\infty}{\frac{\mathcal{L}_{n}(B_{X}(0,R+C))}{\mathcal{L}_{n}(X_{n})}}\geq 1-\epsilon.
$$
Proving that, $0$ belongs to the $(1-\epsilon)$--core. Since $\epsilon$ is arbitrary we finish the proof.
\end{proof}

\begin{cor}
Let $X$ be an infinite hyperbolic graph rooted at $x_{0}$. Let $X_{n}=B(x_0,n)$ the ball of center $x_{0}$ and radius $n$ with the uniform measure. Then $x_0$ belongs to the $\alpha$--core for every $\alpha$.
\end{cor}

\noindent The next results shows that the core is well concentrated for hyperbolic graphs. More specifically, $C_{\alpha}$ is finite for every $\alpha$.

\begin{teo}
Assume that the measure $\mu_{\infty}$ is non--atomic. Then for every $r>0$ and every $\alpha\in [0,1)$, there exists a positive constant $R=R(\alpha,r)$ such that if $x\in X$ and $d(x_{0},x)>R$ then
\begin{equation}
\limsup_{n\to\infty}{\frac{\mathcal{L}_{n}(B_{X}(0,r))}{\mathcal{L}_{n}(X_{n})}}<\alpha.
\end{equation} 
\end{teo}

\begin{proof}
As we did in the previous theorem we can assume that $X$ is quasi--isometrically embedded in $\B^{2}$. Therefore, as we saw previously, there exists a constant $C>0$ such that
$$
\limsup_{n\to\infty}{\frac{\mathcal{L}_{n}(B_{X}(0,r))}{\mathcal{L}_{n}(X_{n})}}\leq \int_{S^{1}}{\mu_{\infty}(E_{u}^{r+C})\,d\mu_{\infty}(u)}.
$$ 
Fix $r>0$ and $\alpha\in [0,1)$. Let $x\in X$ such that $d(x,0)=R$. Rotating the Poincar\'e disk if necessary we can assume that $x=it$ for some $0<t<1$. Let $\phi_{t}$ be the isometry $\phi_{t}:\B^{2}\to B^{2}$ given by 
$$
\phi_{t}(z):=\frac{z-it}{it+z}.
$$
This map takes $x$ to $0$ and extends to an homeomorphism $\phi_{t}:S^{1}\to S^{1}$. Let 
$$
\mu_{\infty}^{(t)}:=\phi_{t}^{*}\mu_{\infty}
$$ 
be the pull--back measure of $\mu_{\infty}$ under this map. This measure is absolutely continuous with respect to $\mu_{\infty}$, and it can be proved easily that its Radon--Nykodim derivative is equal to 
$$
d\mu_{\infty}^{(t)}(\theta)=\frac{1-t^2}{1+t^2+2t\sin(\theta)}\,d\mu_{\infty}(\theta).
$$
If we apply the isometry then the point $x$ goes to $0$ and $\mu_{\infty}$ gets transformed in $\mu_{\infty}^{(t)}$. Therefore,
$$
\limsup_{n\to\infty}{\frac{\mathcal{L}_{n}(B_{X}(x,r))}{\mathcal{L}_{n}(X_{n})}}\leq \int_{S^{1}}{\mu_{\infty}^{(t)}(E_{u}^{r+C})\,d\mu_{\infty}^{(t)}(u)}.
$$ 
As $R$ goes to $\infty$ the value of $t$ increases to $1$ and the measure $\mu_{\infty}^{(t)}$ gets more and more concentrated at the point $-i$. As a matter of fact, the measures $\mu_{\infty}^{(t)}$ converge weakly to the Dirac measure concentrated at $v=-i$ as $t\to 1^{-}$. Hence, given $\epsilon>0$ there exists $R_{0}>0$ such that if $R\geq R_{0}$ then 
$$
\mu_{\infty}^{(t)}(\mathcal{C})>1-\epsilon\quad\text{where}\quad \mathcal{C}=\Big\{e^{i\theta}\,:\,\theta\in\Big(\frac{3\pi}{2}-\epsilon,\frac{3\pi}{2}+\epsilon\Big)\Big\}.
$$
Therefore, for a fixed $r>0$ 
$$
\limsup_{R\to\infty}{\int_{S^{1}}{\mu_{\infty}^{(t)}(E_{u}^{r+C})\,d\mu_{\infty}^{(t)}(u)}}=0.
$$
Then there exists $R=R(\alpha,r)>0$ such that for every $x\in X$ such that $d(x,0)>R$ then 
$$
\limsup_{n\to\infty}{\frac{\mathcal{L}_{n}(B_{X}(x,r))}{\mathcal{L}_{n}(X_{n})}}<\alpha.
$$
\end{proof}

\begin{remark}
Let $\{k_{l}\}_{l=0}^{\infty}$ be a sequence of positive integers and let $T$ be the associated tree as defined in Section \ref{tree}. All the results extend to the general case but for simplicity of the notation let us assume that $k_{l}=k>1$. Therefore, $T$ is a $k$ regular tree.  Let $v$ be any node in the tree and let $l$ be its depth. As we already discussed in Section \ref{tree}, the traffic the asymptotic proportion of the traffic passing through node $v$ is
$$
ap(v)=\frac{1}{k^{l}}\Bigg(2-\frac{1}{k^{l}}-\frac{1}{k^{l+1}} \Bigg).
$$
Therefore, we see that every node in $T$ belongs to $C_{\alpha}$ for some $\alpha$ and hence 
$$
\mathcal{C}=\cup_{\alpha>0}{\,\,C_{\alpha}}=T.
$$
It is clear also that given $\alpha>0$ the set $C_{\alpha}$ is finite and moreover it is equal to $T_{n}$ for some $n$.
\end{remark}

\end{document}